\documentclass[10pt]{amsart}
\usepackage[british,UKenglish]{babel}
\usepackage[utf8]{inputenc}
\usepackage{braket}
\usepackage{xfrac}
\usepackage{amsfonts, amssymb, amsmath, amsthm, amscd}
\usepackage{mathtools}
\usepackage{tikz}
\usepackage{empheq}
\usepackage[left=3.5cm,right=3.5cm,top=3.5cm,bottom=3.5cm]{geometry}
\usepackage{color}
\usepackage{epsfig}  	
\usepackage{epic,eepic}   
\usepackage{setspace}
\usepackage{caption}
\usepackage{graphicx}
\DeclarePairedDelimiterX{\norm}[1]{\lVert}{\rVert}{#1}
\newtheorem{thm}{Theorem}[section]
\newtheorem{thmA}{Theorem}
\newtheorem{corA}[thmA]{Corollary}

\newtheorem{prop}[thm]{Proposition}

\theoremstyle{definition}

\theoremstyle{remark}
\newtheorem{remark}[thm]{Remark}

\DeclareMathOperator{\Vol}{Vol}

\DeclareMathOperator{\Div}{div}

\DeclareMathOperator{\tr}{tr}
\newcommand{\dV}{\, \mathrm{d}V}
\newcommand{\ddt}{\frac{\mathrm{d}}{\mathrm{d}t}}
\begin{document}

\makeatletter
\@namedef{subjclassname@2020}{%
\textup{2020} Mathematics Subject Classification}
\makeatother

\begin{abstract}
We show there are no extremal metrics for the eigenvalues of the Neumann Laplacian on any compact manifold. Nonetheless, we construct examples of conformally extremal metrics for the eigenvalues of this operator in any annulus and characterise these special metrics in the general case of a compact manifold of dimension $n \geq 2$. As for the Dirichlet Laplacian, we prove non existence of extremal metrics on any compact surface.

%We characterise extremal metrics for the first positive eigenvalue of the Laplacian with Neumann boundary condition on %compact Riemannian manifolds with boundary as being induced by minimal immersions into round spheres by first %eigenfunctions. On the other hand, we show there are no extremal metrics for the first eigenvalue of the Laplacian %assuming Dirichlet boundary condition.
\end{abstract}

\title[Extremal metrics for the eigenvalues of the Laplacian]{Extremal metrics for the eigenvalues of the Laplacian on manifolds with boundary}
\author{Eduardo Longa}

\address{Departamento de Matem\'{a}tica, Instituto de Matem\'{a}tica e Estat\'{i}stica, Universidade de S\~{a}o Paulo, R. do Mat\~{a}o 1010, S\~{a}o Paulo, SP 05508-900, Brazil}
\email{eduardo.r.longa@gmail.com}



\maketitle

\let\thefootnote\relax\footnote{The author was partially supported by grants 2021/03599-7 and 2021/09650-4, São Paulo Research Foundation (FAPESP).}

\section{Introduction}

Let $(M^n,g)$ be a closed Riemannian manifold of dimension $n \geq 2$. The Laplace-Beltrami operator $\Delta_g$, or just Laplacian, acts on smooth functions as $\Delta_g f = -\Div \nabla f$. In the case of compact manifolds, its spectrum is discrete and consists entirely of eigenvalues, which we order and list counted with multiplicity:
\begin{align*}
0=\lambda_0(M,g) < \lambda_1(M,g) \leq \lambda_2(M,g) \leq \cdots \nearrow \infty.
\end{align*}

When considered as a functional over the space of Riemannian metrics on $M$, $\lambda_k$ is not scale invariant (in fact, $\lambda_k(M,cg) = \frac{1}{c} \lambda_k(M,g)$ for any constant $c>0$). Consider then the normalised eigenvalues
\begin{align*}
\overline{\lambda}_k(M,g) = \lambda_k(M,g) \Vol(M,g)^{2/n}
\end{align*}

One of the main problems in Spectral Geometry is to compute the exact values of 
\begin{align*}
\Lambda_k(M) = \sup_g \overline{\lambda}_k(M,g),
\end{align*}
and the corresponding maximising metrics. Here the supremum is taken over all smooth Riemannian metrics $g$ on $M$. When $M$ has dimension $n \geq 3$, then $\Lambda_1(M) = \infty$ (see \cite{colbois}; nonetheless, the supremum above is always finite when considered with respect to a conformal class of metrics, as shown in \cite{korevaar}). In the case of surfaces, many authors have contributed to this program by either calculating $\Lambda_k(M)$ and the maximising metrics for some surfaces or giving upper bounds for this quantity. We recommend the excellent survey \cite{karpukhin2021} for the interested reader to check the development of the problem and the main open questions as of the publishing of this article.

A deep and surprising feature of maximising metrics for $\Lambda_1$ in a surface was discovered by Nadirashvili \cite{nadirashvili}: they are arise as induced by minimal immersions of the surface into a round sphere by first eigenfunctions. Later, El Soufi and Ilias (\cite{elsoufi2000}, \cite{elsoufi2008}) generalized this result for closed manifolds of arbitrary dimension using the related concept of extremal metric (see Section \ref{derivative of eigenvalues} for details). Finally, Fraser and Schoen \cite{fraser} showed that maximising metrics for the first normalised Steklov eigenvalue on a compact bordered surface are induced by free boundary minimal immersions of the surface into some Euclidean ball. This result and related techniques have been used to construct many examples of free boundary minimal surfaces in the $3$-ball $\mathbb{B}^3$. Very recently, Karpukhin, Kusner, McGrath and Stern \cite{karpukhin2024} employed eigenvalue optimisation techniques to solve the topological realisation problem for free boundary minimal surfaces in $\mathbb{B}^3$.

\begin{remark}
It is not obvious whether for a given surface $M$ the supremum $\Lambda_1(M)$ is attained by a smooth Riemannian metric. This question was recently addressed by Matthiesen and Siffert \cite{siffert}, where they claim that maximising metrics exist in any closed surface and are smooth away of at most finitely many conical singularities. However, these authors have informed those of \cite{karpukhin2024} that the argument in \cite{siffert} ``contains a nontrivial gap''. 
\end{remark}

In this paper we will address analogous problems in the case of a compact Riemannian manifold with boundary $(M^n,g)$. From now on, let us denote by
\begin{align*}
0 < \lambda_1(M,g) < \lambda_2(M,g) \leq \cdots \leq \lambda_k(M,g) \leq \cdots \nearrow \infty.
\end{align*}
and by
\begin{align*}
0 = \mu_0(M,g) < \mu_1(M,g) \leq \cdots \leq \mu_k(M,g) \leq \cdots \nearrow \infty.
\end{align*}
the spectra of the Laplacian on $M$ with Dirichlet and Neumann boundary conditions, respectively. Let us also write
\begin{align*}
\lambda_k^{\#}(M) = \inf_g \lambda_k(M,g) \Vol(M,g)^{2/n}
\end{align*}
and
\begin{align*}
\mu_k^{\#}(M) = \sup_g \mu_k(M,g) \Vol(M,g)^{2/n}
\end{align*}
where, again, the infimum and the supremum are taken with respect to all smooth Riemannian metrics $g$ on $M$.

Li and Yau proved that, in the case of compact surfaces, $\mu_1^{\#}(M)$ is always finite (see Fact 1, Fact 5 and Theorem 1 in \cite{li}). Our first theorem, shows, in particular, that the supremum defining this quantity is never achieved.

\begin{thmA} \label{main neumann}
There are no extremal metrics for $\mu_k$ on any compact Riemannian manifold with boundary of dimension $n \geq 2$, regardless of the value of $k \geq 1$. In particular, $\mu^{\#}_k(M)$ is never attained by a smooth Riemannian metric on $M$.
\end{thmA}

In another paper, Li and Yau also proved that if $(M^n,g)$ is compact with boundary and has nonnegative Ricci curvature, then $\mu_k(M,g) \Vol(M,g)^{2/n}$ is bounded above by $C k^{2/n}$, where $C > 0$ only depends on $n$. Thus, Theorem \ref{main neumann} implies that there are no maximising metrics for the quantity $\mu_k(M,g) \Vol(M,g)^{2/n}$ when $g$ varies on the set of positive Ricci smooth metrics on $M$.

Our second result, on the other hand, charachterises conformally extremal metrics for the Neumann eigenvalues. Together with the examples in Section \ref{Examples}, this shows that the situation is entirely different when we restrict ourselves to fixed conformal classes. 

\begin{thmA} \label{main neumann conformal}
Let $M^n$ be a compact smooth manifold with boundary, $n \geq 2$. If a Riemannian metric $g$ is conformally extremal for $\mu_k$, $k \geq 1$, then there exist eigenfunctions $u_1, \dots, u_m$ associated with $\mu_k(M,g)$ such that 
\begin{itemize}
\item[(i)] $\sum_{i=1}^m u_i^2 = \frac{1}{\mu_k}$;
\item[(ii)] $\sum_{i = 1}^m \vert \nabla u_i \vert^2 = 1$.
\end{itemize}
Conversely, if there exist eigenfunctions $u_1, \dots, u_m$ associated with $\mu_k(M,g)$ satisfying (i) and (ii), and either $\mu_k(M,g) < \mu_{k+1}(M,g)$ or $\mu_k(M,g) > \mu_{k-1}(M,g)$, then the metric $g$ is conformally extremal for $\mu_k$. 
\end{thmA}

There is a beautiful connection between extremal metrics for the Laplacian in conformal classes and the theory of harmonic maps into spheres, firstly explored in \cite{elsoufi2002, fraser} and further developped in \cite{matthiesen, karpukhin2021, karpukhin2022, karpukhin2022_2}. Let us briefly recall some concepts.

Let $(M^n,g)$ and $(N,h)$ be smooth Riemannian manifolds. A map $u \in W^{1,p}(M,N)$ is called $p$-harmonic if it is a critical point of the $p$-energy functional
\begin{align*}
E^p(u) = \int_M \vert \mathrm{d} u \vert^p \dV_g,
\end{align*} 
where $\vert \mathrm{d} u \vert^2$ is computed with respect to the metrics $g$ and $h$. When $p=2$, we simply say that $u$ is harmonic. 

We are only concerned with the case $p = n$, when $(M,g)$ is compact and when $(N,h)$ is the unit round sphere $\mathbb{S}^{m-1}$. With these hypotheses, it is not difficult to show that $u : (M^n, g) \to \mathbb{S}^{m-1}$ is $n$-harmonic if and only if it is a weak solution of
\begin{align*}
\Div(\vert \nabla u \vert^{n-2} \nabla u) = \vert \nabla u \vert^n \nabla u.
\end{align*} 
Moreover, such maps are $C^{1,\alpha}$ for some $\alpha > 0$ when $n > 2$ and they are smooth when either $n = 2$ or $u$ is nondegenerate, i.e., $\mathrm{d} u \neq 0$ (see \cite{takeuchi} or \cite{matthiesen}).

With this language, Theorem \ref{main neumann conformal} can be rephrased in the following way.

\begin{corA} \label{n-harmonic}
Let $M^n$ be a compact smooth manifold with boundary, $n \geq 2$. If a Riemannian metric $g$ is conformally extremal for $\mu_k$, $k \geq 1$, then there exist an $n$-harmonic map $u : (M,g) \to \mathbb{S}^{m-1}$ into the unit round sphere with $\vert \mathrm{d} u \vert^2 = \mu_k(g)$, all of whose components satify the Neumann boundary condition.

Conversely, if there is a nondegenerate harmonic map $u : (M,g) \to \mathbb{S}^{m-1}$ whose components satisfy the Neumann boundary condition, then the metric $\tilde{g} = \vert \mathrm{d} u \vert_g^2 \, g$ is conformally extremal for $\mu_k$, where $k$ is the smallest positive integer index associated to the eigenvalue $1$. 
\end{corA}

As an application of Corollary \ref{n-harmonic}, in Section \ref{Examples} we construct examples of conformally extremal metrics in an annulus by considering appropriately chosen domains in unduloids. 

As for the eigenvalues of the Laplacian with Dirichlet boundary condition, we show nonexistence of extremal metrics both globally and when restricted to conformal classes. 

\begin{thmA} \label{main dirichlet}
There are no extremal metrics for $\lambda_k$ on any compact Riemannian surface with boundary, irrespective of the value of $k \geq 1$. In particular, $\lambda^{\#}_k(M)$ is never attained by a smooth Riemannian metric on $M^2$.
\end{thmA}

\begin{thmA} \label{main dirichlet conformal}
There are no conformally extremal metrics for $\lambda_k$ on any compact Riemannian surface with boundary,  irrespective of the value of $k \geq 1$.
\end{thmA}

Regarding Theorems \ref{main dirichlet} and \ref{main dirichlet conformal}, we don't know what can be said for manifolds of dimension $n \geq 3$. The arguments involved in the proofs can only handle the surface case. In light of Theorems \ref{main neumann} and \ref{main dirichlet}, another relevant question is whether $\lambda^{\#}_k(M)$ and $\mu^{\#}_k(M)$ can be attained by Riemannian metrics with finitely many singularities. This problem seems to be open even in the case of closed surfaces and has been proved only under the assumption of a certain spectral gap (see Theorem 2 in \cite{petrides}).

\section*{Acknowledgements}

The author would like to thank Pieralberto Sicbaldi, Leonardo Bonorino and Paolo Piccione for many fruitful and enlightening mathematical conversations. He also thanks the Institute of Mathematics of the University of Granada (IMAG), where the first ideas of this work were conceived, for the hospitality and excellent working conditions. Finally, the thanks M. Karpukhin for the suggestion on how to construct the examples and A. M\'etras for useful comments concerning the proof of Proposition \ref{derivative lambda}. The author was partially supported by grants 2021/03599-7, 2021/09650-4 and 2024/01663-8, São Paulo Research Foundation (FAPESP).

\section{Extremal metrics and the derivative of the eigenvalues} \label{derivative of eigenvalues}

In this section we introduce the definition of extremal metrics for the eigenvalues of the Laplacian, inspired by the works of Nadirashvili \cite{nadirashvili} and Kapurkhin-M\'etras \cite{karpukhin2022}. Then, we compute the derivative of these eigenvalues and show that, in the case of extremal metrics, a certain induced quadratic form cannot be definite on the corresponding eigenspaces.

Let $M^n$ be a compact smooth manifold with boundary. We say that a Riemannian metric $g$ is \textbf{$F$-extremal} for some functional $F$ if for all $1$-parameter smooth family of Riemannian metrics $g(t)$ with $g(0) = g$, we have either
\begin{align*}
F(g(t)) \leq F(g) + o(t) \quad \text{ or } \quad F(t) \geq F(g) + o(t)
\end{align*}
as $t \to 0$. A metric $g$ is \textbf{$F$-conformally extremal} if it is $F$-extremal in its conformal class $[g]$ (that is, we require $g(t) \in [g]$ above).

To unify notation, let $\Lambda_k(g)$ denote either $\lambda_k(M,g)$ or $\mu_k(M,g)$, for $k \geq 1$. Also, let $E_k(g)$ be the corresponding eigenspace in $L^2(M,g)$. The functional $F$ we are interested in is 
\begin{align*}
F(g) = \Lambda_k(g) \Vol(M,g)^{2/n},
\end{align*} 
defined on the space of all smooth Riemannian metrics on $M$.

It is possible to prove that this functional is Lipschitz with respect to the $C^\infty$ topology in the space of Riemannian metrics (see \cite{kriegl}, for instance). The next proposition gives a formula for its derivative along any smooth family of metrics.

\begin{prop} \label{derivative lambda}
Let $(-\varepsilon, \varepsilon) \ni t \mapsto g(t)$ be a smooth $1$-parameter family of Riemannian metrics on $M$. Then the map $t \mapsto \Lambda_k(t) := \Lambda_k(g(t))$ is Lipschitz. Moreover, for an open and dense set $\mathcal{O} \subseteq (-\varepsilon, \varepsilon)$, if $t_0 \in \mathcal{O}$ and $\dot{\Lambda}_k(t_0)$ exists, then
\begin{align*}
\dot{\Lambda}_k(t_0) = - \int_M \langle q(u), h \rangle \dV_{g_{t_0}},
\end{align*}
for any $u \in E_k(g(t_0))$ with $\norm{u}_{L^2(M, g(t_0))} = 1$. Here, $h = \ddt g(t) \vert_{t = t_0}$, $\langle \cdot, \cdot \rangle$ is the inner product  induced by $g(t_0)$ in the space $S^2(M)$
of smooth covariant $2$-tensors, $\dV_{g_{t_0}}$ is the volume element of $M$ induced by $g(t_0)$, and $q : C^{\infty}(M) \to S^2(M)$ is given by
\begin{align*}
q(f) = df \otimes df - \frac{1}{2}\left( \vert \nabla f \vert^2 - \Lambda_k(t_0) f^2 \right) g(t_0).
\end{align*}
\end{prop}

\begin{proof}
For each $i \geq 2$, let 
\begin{align*}
A_i = \{t \in (-\varepsilon, \varepsilon) : \Lambda_i(t) > \Lambda_{i-1}(t) \}
\end{align*}
and let $B_i$ be the interior of $(-\varepsilon, \varepsilon) \setminus A_i$, so that
\begin{align*}
(-\varepsilon, \varepsilon) = A_i \cup B_i \cup \partial A_i
\end{align*}
is a disjoint union. Let $\mathcal{O}$ be the complement of $\bigcup_{i=2}^k \partial A_i$ in $(-\varepsilon, \varepsilon)$. Since each $\partial A_i$ is closed and has empty interior, $\mathcal{O}$ is open and dense. It is not difficult to see that for any $t_0 \in \mathcal{O}$, there is $1 \leq l \leq k$ such that $\Lambda_l(t) = \Lambda_k(t)$ and $\Lambda_{l-1}(t) < \Lambda_l(t)$ in a neighbourhood around $t_0$. Fix any $t_0 \in\mathcal{O}$ where $\dot{\Lambda}_k(t_0)$ exists and define on this neighbourhood
\begin{align*}
\mathcal{E}(t) = \begin{cases}
\bigcup_{i=1}^{l-1} E_i(g(t)), \text{if } \Lambda_k(t) = \lambda_k(t) \vspace{0.1cm}  \\ 
\bigcup_{i=0}^{l-1} E_i(g(t)), \text{if } \Lambda_k(t) = \mu_k(t)
\end{cases}.
\end{align*} 
Then $\mathcal{E}(t)$ has constant dimension $l-1$ in the Dirichlet case, or $l$ in the Neumann case. Moreover, $\mathcal{E}(t)$ varies smoothly in $t$ for small $t$. Let $P_t : L^2(M, g(t)) \to \mathcal{E}(t)$ be the orthogonal projection onto $\mathcal{E}(t)$. Choose $u_0 \in E_k(g(t_0))$ with $\norm{u_0}_{L^2(M, g(t_0))} = 1$ and let $u_t = u_0 - P_t(u_0)$. Then, the functional 
\begin{align*}
f(t) = \int_M \vert \nabla u_t\vert^2 \dV_t - \Lambda_l(t) \int_M u_t^2 \dV_t
\end{align*}
satisfies $f(t) \geq 0$ on this neighbourhood of $t_0$ and $f(t_0) = 0$, so that $\dot{f}(t_0) = 0$. It remains to compute this derivative. We have
\begin{align*}
\dot{f}(t_0) &= \int_M \left( \ddt \Big|_{t=t_0} \langle \nabla u_t, \nabla u_t \rangle \right) \dV_{g_{t_0}} + \int_M \vert \nabla u_0 \vert^2 \left( \ddt \Big|_{t = t_0} \dV_t \right) \\
&\hphantom{zz}+ \dot{\Lambda}_l(t_0) + \Lambda_l(t_0) \int_M \left( \ddt \Big|_{t = t_0} u_t^2 \right) \dV_{g_{t_0}} + \Lambda_l(t_0) \int_M u_0^2 \left( \ddt \Big|_{t = t_0} \dV_t \right) \\
&=\int_M \left[ 2 \langle \nabla u_0, \nabla \dot{u}_0\rangle  - h(\nabla u_0, \nabla u_0) + \frac{1}{2} \vert \nabla u_0 \vert^2 \tr_{g(t_0)} h \right] \dV_{g_{t_0}} \\
&\hphantom{zz}+ \dot{\Lambda}_l(t_0) + \Lambda_l(t_0) \int_M \left( 2u_0 \dot{u}_0  + \frac{1}{2} u_0^2 \tr_{g(t_0)} h \right) \dV_{g_{t_0}} \\
&= \int_M \left[ 2 \langle \nabla u_0, \nabla \dot{u}_0\rangle - \langle du_0 \otimes du_0 - \frac{1}{2} \vert \nabla u_0 \vert^2 g(t_0), h \rangle \right] \dV_{g_{t_0}}  \\
&\hphantom{zz}+ \dot{\Lambda}_l(t_0) + \Lambda_l(t_0) \int_M \left( 2u_0 \dot{u}_0  + \frac{1}{2} u_0^2 \langle g(t_0), h \rangle \right) \dV_{g_{t_0}} 
\end{align*}
Since $u_0 \in E_l(g(t_0))$, we also have
\begin{align*}
\int_M \langle \nabla u_0, \nabla \dot{u}_0 \rangle \dV_{g_{t_0}} = \Lambda_l(t_0) \int_M u_0 \dot{u}_0 \dV_{g_{t_0}}.
\end{align*}
Thus, rearranging the formula for $\dot{f}(t_0)$, we obtain
\begin{align*}
\dot{\Lambda}_l(t_0) = - \int_M  \langle du_0 \otimes du_0 - \frac{1}{2}\left( \vert \nabla u_0 \vert^2 - \Lambda_l(t_0) u_0^2 \right) g(t_0), h \rangle \dV_{g_{t_0}},
\end{align*}
as we wanted.
\end{proof}

\begin{remark}
Since the map $t \mapsto \Lambda_k(g(t))$ is Lipschitz, it is differentiable almost everywhere, say, in a set $\mathcal{D} \subseteq (-\varepsilon, \varepsilon)$ of total measure. Thus, the formula stated in Proposition \ref{derivative lambda} for the derivative of this map is valid for all $t$ in the dense set $\mathcal{O} \cap \mathcal{D}$.
\end{remark}

Given a Riemannian metric $g$ on $M$, let $L^2(S^2(M),g)$ denote the space of $L^2$ symmetric covariant $2$-tensors on $M$. For $h \in L^2(S^2(M),g)$, let $Q_h : C^\infty(M) \to \mathbb{R}$ be the quadratic form defined by
\begin{align*}
Q_h(u) = - \int_M \langle q(u), h \rangle \dV_g,
\end{align*}
where $q$ is given in Proposition \ref{derivative lambda}.

\begin{prop} \label{isotropic}
If a metric $g$ is $\Lambda_k$-extremal, $k \geq 1$, then for any $h \in L^2(S^2(M),g)$ which satisfies $\int_M \langle g, h \rangle \dV_g = 0$, there exists $u \in E_k(g)$ with $\norm{u}_{L^2(M,g)} = 1$ such that $Q_h(u) = 0$.
\end{prop}

\begin{proof}
Since $S^2(M)$ is dense in $L^2(S^2(M),g)$, there exists a sequence $h_j \in C^2(S^2(M),g)$ such that $\int_M \langle g, h_j \rangle \dV_g = 0$ for any $j \geq 1$ and $h_j \to h$ in $L^2(S^2(M),g)$ as $j \to\infty$.

For each $j \geq 1$, let $(-a_j, a_j) \ni t \mapsto g_j(t)$ be the smooth $1$-parameter family of metrics defined by
\begin{align*}
g_j(t) = \frac{\Vol(M,g)}{\Vol(M, g + th_j)^{2/n}}(g + th_j),
\end{align*}
for $t \in (-a_j, a_j)$. Then $\Vol(M, g_j(t)) = 1$ for all $t$, $g_j(0) = g$ and $\ddt g_j(t)|_{t = 0} = h_j$. Since $g$ is $\Lambda_k$-extremal, we can assume without loss of generality that
\begin{align*}
\Lambda_k(g_j(t)) \leq \Lambda_k(g) + o(t)
\end{align*}
as $t \to 0$. In particular, taking the limit for negative $t$ yields
\begin{align*}
\lim_{t \to 0^-} \frac{\Lambda_k(g_j(t)) - \Lambda_k(g)}{t} \geq 0.
\end{align*}
So, there exists a sequence of $\varepsilon_i > 0$ decreasing to $0$ and $\delta_i \in \mathbb{R}$ with $\lim_{i \to \infty} \delta_i = 0$ such that
\begin{align*}
\delta_i \leq \frac{\Lambda_k(g_j(-\varepsilon_i)) - \Lambda_k(g)}{\varepsilon_i} = \frac{1}{\varepsilon_i} \int_{-\varepsilon_i}^0 \dot{\Lambda}_k(g_j(t)) \, \mathrm{d}t \leq \operatorname{ess \, sup} \{   \dot{\Lambda}_k(t) : t \in [-\varepsilon_i, 0] \}. 
\end{align*}
We can thus find a sequence of $t_i < 0$ increasing to $0$ such that $\dot{\Lambda}_k(g_j(t_i))$ exists and is given by 
\begin{align*}
\dot{\Lambda}_k(g_j(t_i)) = Q_{\omega_{ij}}(u^{(j)}_i) \geq \delta_i,
\end{align*}
where $\omega_{ij} = \ddt g_j(t)|_{t = t_i}$ and $u^{(j)}_i \in E_k(g_j(t_i))$ with $\norm{u^{(j)}_i}_{L^2(M, g_j(t_i))} = 1$ (see Proposition \ref{derivative lambda}). Passing to a subsequence if necessary, $u^{(j)}_i \to u^{(j)}_-$ in $C^2(M)$ as $i \to \infty$. Then $u^{(j)}_- \in E_k(g_j(0)) = E_k(g)$, $\norm{u^{(j)}_-}_{L^2(M,g)} = 1$ and
\begin{align*}
Q_{h_j}(u^{(j)}_-) = \lim_{i \to \infty} Q_{\omega_{ij}}(u^{(j)}_-) \geq \lim_{i \to \infty} \delta_i \geq 0.
\end{align*}
Again, taking a subsequence if necessary, $u^{(j)}_- \to u_-$ in $C^2(M)$ as $j \to \infty$, and $u_- \in E_k(g)$, $\norm{u_-}_{L^2(M,g)} = 1$. Moreover, 
\begin{align*}
Q_h(u_-) = \lim_{j \to \infty}Q_{h_j}(u^{(j)}_-) \geq 0.
\end{align*}
The same argument starting with the limit for positive values of $t$ produces $u_+ \in E_k(g)$ with $\norm{u_+}_{L^2(M,g)} = 1$. Since $Q_h$ is continuous and takes both nonpositive and nonnegative values on $E_k(g)$, there must exist $u \in E_k(g)$ such that $Q_h(u) = 0$, as claimed.
\end{proof}

The next two results, which deal with deformations of a metric in a fixed conformal class, can be proved in an analogous fashion.

\begin{prop} \label{derivative lambda conformal}
Let $(-\varepsilon, \varepsilon) \ni t \mapsto g(t) = e^{\varphi(t)} g$ be a smooth $1$-parameter family of Riemannian metrics on $M$, conformal to some fixed metric $g$. Then the map $t \mapsto \Lambda_k(t) := \Lambda_k(g(t))$ is Lipschitz. Moreover, for an open and dense set $\mathcal{O} \subseteq (-\varepsilon, \varepsilon)$, if $t_0 \in \mathcal{O}$ and $\dot{\Lambda}_k(t_0)$ exists, then
\begin{align*}
\dot{\Lambda}_k(t_0) = - \int_M  p(u) \dot{\varphi}(t_0) \dV_{g_{t_0}},
\end{align*}
for any $u \in E_k(g(t_0))$ with $\norm{u}_{L^2(M, g(t_0))} = 1$. Here, $p : C^{\infty}(M) \to C^{\infty}(M)$ is given by
\begin{align*}
p(f) = \frac{(2-n)}{2} \vert \nabla f \vert^2 + \frac{n \Lambda_k(t_0)}{2} f^2.
\end{align*}
\end{prop}

Given a Riemannian metric $g$ on $M$ and a function $\psi \in L^2(M, g)$ , let $P_\psi : C^\infty(M) \to \mathbb{R}$ be the quadratic form defined by
\begin{align*}
P_\psi(u) = - \int_M  p(u) \psi \dV_g,
\end{align*}
where $p$ is given in Proposition \ref{derivative lambda conformal}.

\begin{prop} \label{isotropic conformal}
If a metric $g$ is conformally $\Lambda_k$-extremal, $k \geq 1$, then for any $\psi \in L^2(M,g)$ which satisfies $\int_M \psi \dV_g = 0$, there exists $u \in E_k(g)$ with $\norm{u}_{L^2(M,g)} = 1$ such that $P_\psi(u) = 0$.
\end{prop}

\section{Proofs of the Theorems} \label{Proofs}

In this section we present the proofs of the theorems stated in the Introduction.

\begin{proof}[Proof of Theorem \ref{main neumann}]
Let $M^n$ be a compact manifold with boundary and suppose $g$ is a smooth Riemannian metric which is extremal for $\mu_k$ for some $k \geq 1$. Let $K$ be the convex hull in $L^2(S^2(M),g)$ of the set $\{ q(u) : u \in E_k(g) \}$. We claim that $g \in K$. If not, then since $K$ is a convex cone which lies in a finite dimensional subspace of $L^2(S^2(M),g)$, the Hahn-Banach theorem furnishes $h \in L^2(S^2(M),g)$ such that
\begin{align*}
\int_M \langle g, h \rangle \dV_g > 0 \quad \text{ and } \quad \int_M \langle q(u), h \rangle \dV_g < 0 \quad \text{for all } u \in E_k(g) \setminus \{0\}.
\end{align*}
Let $\tilde{h}$ be the projection of $h$ onto the subspace of symmetric covariant $2$-tensors whose traces have zero mean. Explicitly, define
\begin{align*}
\tilde{h} = h - \left( \frac{\int_M \langle g, h \rangle \dV_g}{n \Vol(M,g)} \right) g.
\end{align*}
Then $\int_M \langle g, \tilde{h} \rangle \dV_g = 0$ and for any $u \in E_k(g)$ we have
\begin{align*}
Q_{\tilde{h}}(u) &= \int_M \langle q(u), \tilde{h} \rangle \dV_g = \int_M \langle q(u), h \rangle \dV_g  - \left( \frac{\int_M \langle g, h \rangle \dV_g}{n \Vol(M,g)} \right) \int_M \langle q(u), g \rangle\dV_g \\
&= \int_M \langle q(u), h \rangle \dV_g + \left( \frac{\int_M \langle g, h \rangle \dV_g}{n \Vol(M,g)} \right) \int_M \left[ \vert \nabla u \vert^2 - \frac{n}{2} \left( \vert \nabla u \vert^2 - \mu_k(g) u^2 \right) \right] \dV_g \\
&= \underbrace{\int_M \langle q(u), h \rangle \dV_g}_{ \leq 0} - \underbrace{\left( \frac{\int_M \langle g, h \rangle \dV_g}{n \Vol(M,g)} \right)}_{ > 0} \underbrace{\int_M \vert \nabla u \vert^2 \dV_g}_{ > 0} < 0,
\end{align*}
which contradicts Proposition \ref{isotropic}. Thus $g \in K$ and there exist independent eigenfunctions $u_1, \dots, u_m \in E_k(g)$ such that 
\begin{align} \label{equation mu_k}
\sum_{i = 1}^m \left[ du_i \otimes du_i  - \frac{1}{2} \left( \vert \nabla u_i \vert^2 - \mu_k(g) u_i^2 \right)g \right] = g.
\end{align}
Taking the trace of this equation and rearranging, we obtain
\begin{align} \label{equation2 mu_k}
\frac{\mu_k(g)}{2} \sum_{i=1}^{m} u_i^2 = 1 + \frac{(n-2)}{2n} \sum_{i=1}^{m} \vert \nabla u_i \vert^2
\end{align}

Let $u = (u_1, \dots, u_m) : M \to \mathbb{R}^{m}$. If $n= 2$, equation (\ref{equation2 mu_k}) immediately gives that $\vert u \vert^2 = \frac{2}{\mu_k(g)}$ and then equation (\ref{equation mu_k}) shows that $g = \sum_{i=1}^{m} \mathrm{d}u_i \otimes \mathrm{d}u_i $. If $n \geq 3$, let $f = \vert u \vert^2 - \frac{n}{\mu_k(g)}$. After some simple calculations we see that $f$ satisfies the following boundary value problem:
\begin{align*}
\begin{cases}
\Delta_g f = -\frac{4\mu_k(g)}{n-2} f, \quad \text{on } M \\
\langle \nabla f, \nu \rangle = 0, \quad \text{on } \partial M
\end{cases},
\end{align*}
where $\nu$ is the outward unit normal along $\partial M$. Since the Neumann Laplacian is a positive self-adjoint operator, it follows that $f$ vanishes. So, $\vert u \vert^2 = \frac{n}{\mu_k(g)}$. As before, equation (\ref{equation mu_k}) shows that $g = \sum_{i=1}^{m} \mathrm{d}u_i \otimes \mathrm{d}u_i$. 

In any case, $u$ is an isometric immersion from $(M,g)$ into the round sphere of radius $\sqrt{\frac{n}{\mu_k(g)}}$ around the origin of $\mathbb{R}^m$. But the normal derivative of $u$ along the boundary of $M$ vanishes since all component of $u$ satisfy a Neumann boundary condition. This contradiction shows that $g$ cannot be extremal for $\mu_k$.   
\end{proof}

\begin{proof}[Proof of Theorem \ref{main dirichlet}]
The argument is similar to that of the proof of Theorem \ref{main neumann}. Exactly as in that proof, we can show that $g$ is a convex combination of Dirichlet eigenfunctions $u_1, \dots, u_m \in E_k(g)$:
\begin{align} \label{equation lambda_k}
\sum_{i = 1}^m \left[ du_i \otimes du_i  - \frac{1}{2} \left( \vert \nabla u_i \vert^2 - \lambda_k(g) u_i^2 \right)g \right] = g.
\end{align}
Again, taking the trace we obtain
\begin{align} \label{equation2 lambda_k}
\frac{\lambda_k(g)}{2} \sum_{i=1}^{m} u_i^2 = 1 + \frac{(n-2)}{2n} \sum_{i=1}^{m} \vert \nabla u_i \vert^2
\end{align}

Let $u = (u_1, \dots, u_m) : M \to \mathbb{R}^{m}$. As $\dim M = n= 2$, equation (\ref{equation2 lambda_k}) immediately yields $\vert u \vert^2 = \frac{2}{\lambda_k(g)}$. This contradicts the fact that $u$ vanishes along the boundary of $M$ since all its components satisfy a Dirichlet boundary condition.
\end{proof}

\begin{remark}
If $n \geq 3$,  some simple calculations show that the function $f = \vert u \vert^2$ in the proof of Theorem \ref{main dirichlet} satisfies the following boundary problem:
\begin{align*}
\begin{cases}
\Delta_g f = -\frac{4\lambda_k(g)}{n-2} f + \frac{4n}{n-2}, \quad \text{on } M, \\
f = 0, \quad \text{on } \partial M
\end{cases}
\end{align*}
This problem has a unique solution in $C^\infty(M)$ by the Fredholm alternative. Thus, there is no contradiction in this case and nothing can be concluded from this argument.
\end{remark}

We now prove Theorem \ref{main neumann conformal} and Theorem \ref{main dirichlet conformal}. 

\begin{proof}[Proof of Theorem \ref{main dirichlet conformal}]
Let $M^2$ be a compact surface with boundary and suppose $g$ is a smooth Riemannian metric which is conformally extremal for $\lambda_k$ for some $k \geq 1$. By considering $K$ the convex hull in $L^2(M,g)$ of the set $\{ p(u) : u \in E_k(g) \}$ and reasoning as in the proof of Theorem \ref{main neumann}, we can show that $1 \in K$. So, there exist independent eigenfunctions $u_1, \dots, u_m \in E_k(g)$ such that
\begin{align} 
\frac{(2-n)}{2} \sum_{i=1}^m \vert \nabla u_i \vert^2 + \frac{n \lambda_k(g)}{2} \sum_{i = 1}^m u_i^2 = 1.
\end{align}
Since $n = 2$, this equation implies that $\sum_{i = 1}^m u_i^2 = \frac{1}{\lambda_k(g)}$, which contradicts the fact that all $u_i$ vanish along $\partial M$.
\end{proof}

\begin{proof}[Proof of Theorem \ref{main neumann conformal}]
Let $M^n$ be a compact manifold with boundary and suppose $g$ is a smooth Riemannian metric which is conformally extremal for $\mu_k$ for some $k \geq 1$. By considering $K$ the convex hull in $L^2(M,g)$ of the set $\{ p(u) : u \in E_k(g) \}$ and reasoning as in the proof of Theorem \ref{main neumann}, we can show that $1 \in K$. So, there exist independent eigenfunctions $u_1, \dots, u_m \in E_k(g)$ such that
\begin{align} \label{sum}
\frac{(2-n)}{2} \sum_{i=1}^m \vert \nabla u_i \vert^2 + \frac{n \mu_k(g)}{2} \sum_{i = 1}^m u_i^2 = 1.
\end{align}

We claim that $\vert u \vert^2 = \frac{1}{\mu_k(g)}$. This is clear if $n=2$, so suppose $n \geq 3$. Notice that 
\begin{align} \label{laplacian product}
\Delta_g \vert u_i \vert^2 = 2 \mu_k(g) \vert u_i \vert^2 - 2 \vert \mathrm{d} u_i \vert^2. 
\end{align}
for any $i=1, \dots, m$. Combining (\ref{laplacian product}) and (\ref{sum}), we see that the map $f = \vert u \vert^2 - \frac{1}{\mu_k(g)}$ satisfies the following boundary value problem:
\begin{align}
\begin{cases}
\Delta_g f = -\frac{4 \mu_k(g)}{n-2} f, \quad \text{on } M\\
\langle \nabla f, \nu \rangle = 0, \quad \text{on } \partial M
\end{cases}.
\end{align}
By the positivity of $\Delta_g$ with Neumann boundary condition, we conclude that $f$ vanishes, i.e., $\vert u \vert^2 = \frac{1}{\mu_k(g)}$. In any case, (\ref{sum}) implies that $\vert \mathrm{d} u \vert^2 = 1$, as we wanted.

Conversely, suppose that $u_1, \dots, u_m \in E_k(g)$ satisfy hypotheses (i) and (ii) and suppose $\mu_k(g) > \mu_{k-1}(g)$ (the other case is treated similarly). Let $g(t) = e^{\varphi(t)}g$ be a smooth $1$-parameter family of Riemannian metrics with $g(0) = 0$ and $\Vol(M,g(t)) = \Vol(M,g)$, so that $\int_M \dot{\varphi} \dV_g = 0$. We have:
\begin{align*}
\sum_{i=1}^m P_{\dot{\varphi}}(u_i) &= - \sum_{i = 1}^m \int_M p(u_i) \dot{\varphi} \dV_g = \int_M \left( \frac{n-2}{2} \sum_{i=1}^m \vert \nabla u_i \vert^2  - \frac{n \mu_k(g)}{2} \sum_{i=1}^m u_i^2 \right) \dot{\varphi} \dV_g \\
&= \int_M \left( \frac{n-2}{2} - \frac{n}{2} \right) \dot{\varphi} \dV_g = 0.
\end{align*}
Thus, letting $E = \operatorname{span} \{u_1, \dots, u_m\}$, we see that there exist $u_{+}, u_{-} \in E$ with $\pm P_{\dot{\varphi}}(u_\pm) \leq 0$. Since $\mu_k(g) > \mu_{k-1}(g)$, $P_{\dot{\varphi}}(u_+) \leq 0$ gives that $\lim_{t \to 0^+} \frac{\mu_k(g(t)) - \mu(g)}{t} \leq 0$ and $P_{\dot{\varphi}}(u_+) \leq 0$ implies that $\lim_{t \to 0^-} \frac{\mu_k(g(t)) - \mu(g)}{t} \geq 0$. This proves that $\mu_k(g(t)) \leq \mu_k(g) + o(t)$ as $t \to 0$. Since $g(t)$ was arbitrary, we conclude that $g$ is conformally extremal for $\mu_k$, as desired.

\section{Examples} \label{Examples}

To conclude, we present examples of conformally extremal metrics for the Neumann eigenvalues in an annulus. The idea is the following: we first find an appropriate domain in the unduloid in $\mathbb{R}^3$ for which the normal derivative of the Gauss map vanishes along the boundary. Since the unduloid has constant mean curvature (CMC), a well-known theorem of Ruh-Vilms \cite{ruh} guarantees that the Gauss map is (nondegenerate) harmonic. We can thus apply Corollary \ref{n-harmonic} to conclude that there is a conformal factor for which the corresponding metric is conformally extremal for the eigenvalue $1$. Let us proceed to the formalities.

Let $x, z \in C^2(\mathbb{R})$, with $x > 0$, and consider the surface of revolution $M$ obtained by rotating the curve $C : v \mapsto (x(v), z(v))$ in the $xz$ plane around the $z$ axis. A parametrization of $M$ is given by $\varphi : \mathbb{R} \times [0, 2\pi]$, defined by
\begin{align*}
\varphi(u,v) = (x(v) \cos u, x(v) \sin u, z(v)),
\end{align*}
and the coefficients of metric $g$ on $M$ induced by $\mathbb{R}^3$ are:
\begin{align*}
E(u,v) &= \left\langle \frac{\partial \varphi}{\partial u}(u,v), \frac{\partial \varphi}{\partial u}(u,v) \right\rangle = x(v)^2  \\
F(u,v) &= \left\langle \frac{\partial \varphi}{\partial u}(u,v), \frac{\partial \varphi}{\partial v}(u,v) \right\rangle = 0\\
G(u,v) &= \left\langle \frac{\partial \varphi}{\partial v}(u,v), \frac{\partial \varphi}{\partial v}(u,v) \right\rangle = x'(v)^2 + z'(v)^2.
\end{align*}
A unit normal vector field $N$ for $M$ is
\begin{align*}
N(u,v) = \frac{1}{\sqrt{G(u,v)}}(z'(v) \cos u, z'(v) \sin u, -x'(v)). 
\end{align*}

Given $a < b$, consider the domain 
\begin{align*}
\Sigma_{a,b} = \{\varphi(u, v) : u \in [0, 2 \pi], v \in [a, b] \} \subset M.
\end{align*}
The outward unit normals for $\partial \Sigma_{a,b}$ are 
\begin{align*}
\nu_a(u) = -\frac{1}{\sqrt{G(u,a)}} \frac{\partial \varphi}{\partial v}(u, a) = - \frac{1}{\sqrt{G(u,a)}}(x'(a) \cos u, x'(a) \sin u, z'(a)) \\
\end{align*}
and
\begin{align*}
\nu_b(u) = \frac{1}{\sqrt{G(u,b)}}\frac{\partial \varphi}{\partial v}(u, b) = \frac{1}{\sqrt{G(u,b)}}(x'(b) \cos u, x'(b) \sin u, z'(b)) \\
\end{align*}
Recall that the signed curvature $k$ of the curve $C$ is given by the formula
\begin{align*}
k(v) = \frac{x'(v)z''(v) - x''(v)z'(v)}{\sqrt{x'(v)^2 + z'(v)^2}}.
\end{align*}
After some straightforward computations, we see that the normal derivatives of $N$ along $\partial \Sigma_{a,b}$ are
\begin{align*}
\begin{cases}
\displaystyle \nu_a(N)(u) = -\frac{1}{\sqrt{G(u,a)}} \frac{\partial N}{\partial v}(u,a) = -\frac{k(a)}{\sqrt{G(u,a)}} \frac{\partial \varphi}{\partial v}(u,a) \vspace{0.2cm}\\
\displaystyle \nu_b(N)(u) = \frac{1}{\sqrt{G(u,b)}} \frac{\partial N}{\partial v}(u,b) = \frac{k(b)}{\sqrt{G(u,b)}} \frac{\partial \varphi}{\partial v}(u,b)
\end{cases}.
\end{align*} 
Thus, if  the curvature $k(v)$ vanishes at $v = a$ and $v = b$, then the components of $N$ satisfy the Neumann boundary condition in $\Sigma_{a,b}$.

Let now $M$ be an unduloid in $\mathbb{R}^3$, whose generating curve $C$ is obtained \cite{hadzhilazova, mladenov} by considering
\begin{align*}
x(v) &= \sqrt{\beta \sin(\mu v) + \delta} \\
z(v) &= \alpha F \left( \frac{\mu v}{2} - \frac{\pi}{4}, \kappa \right)+ \gamma E \left( \frac{\mu v}{2} - \frac{\pi}{4}, \kappa \right),
\end{align*}
where $0 < \alpha < \gamma$ are parameters, $F(\varphi, k)$ and $E(\varphi, k)$ are the elliptic integrals of the first and second kind, respectively, and 
\begin{align*}
\mu = \frac{2}{\alpha + \gamma}, \quad \kappa^2 = \frac{\gamma^2 - \alpha^2}{\gamma^2}, \quad \beta = \frac{\gamma^2 - \alpha^2}{2}, \quad \delta = \frac{\gamma^2 + \alpha^2}{2}.
\end{align*}

It is clear that $C$ is periodic of a certain period $T = T(\alpha, \gamma)$ and its curvature $k$ vanishes twice in the interval $[0,T]$, say at $v = v^\ast$ and $v = v^{\ast \ast}$ (see Figure \ref{unduloid}). Let $v^\ast_n = v^\ast + nT$ and $v^{\ast\ast}_n = v^{\ast\ast} + n T$, $n \in \mathbb{Z}$, be all the vanishing points of $k$. By what we have seen before, the components of the Gauss map $N$ satisfy the Neumann boundary condition in the domains
\begin{align*}
\Sigma^\ast_{n} &= \{ \varphi(u, v) : u \in [0,2\pi], v \in [v^\ast, v^{\ast}_n]\}, \quad n \geq 1 \\
\Sigma^{\ast\ast}_{n} &= \{ \varphi(u, v) : u \in [0,2\pi], v \in [v^\ast, v^{\ast\ast}_n]\}, \quad n \geq 0
\end{align*}

Moreover, the expressions for the principal curvatures of the unduloid \cite{hadzhilazova} are
\begin{align*}
k_1(v) &= \frac{(\gamma - \alpha)\left[ \gamma - \alpha  + (\alpha + \gamma) \sin \left( \frac{2v}{\alpha + \gamma} \right)\right]}{(\alpha + \gamma) \left[ \alpha^2 + \gamma^2 + (\gamma^2 - \alpha^2) \sin \left( \frac{2v}{\alpha  +\gamma} \right)\right]} \\ 
k_2(v) &= \frac{\alpha + \gamma + (\gamma - \alpha) \sin \left( \frac{2v}{\alpha + \gamma} \right) }{\alpha^2 + \gamma^2 + (\gamma^2 - \alpha^2) \sin \left( \frac{2v}{\alpha + \gamma} \right)}.
\end{align*} 
It is easy to see that $k_2(v) \geq \frac{\alpha}{\gamma^2} > 0$. In particular, $\mathrm{d} N \neq 0$ everywhere. Thus, $N : \Sigma^\ast_{n} \to \mathbb{S}^2$ and $N : \Sigma^{\ast\ast}_{n} \to \mathbb{S}^2$ are nondegenerate harmonic maps whose components satisfy the Neumann condition along the respective boundary. Applying Corollary \ref{n-harmonic}, we conclude that the metric $\tilde{g} = \vert \mathrm{d} N \vert^2 g$ on the annuli $[v^\ast,v^\ast_n] \times \mathbb{S}^1$ ($n \geq 1$) and $[v^\ast,v^{\ast\ast}_n] \times \mathbb{S}^1$ ($n \geq 0$) are conformally extremal for the Neumann eigenvalue $1$. Varying $n$, we obtain infinitely many such metrics, no two of which are conformally equivalent.

\begin{figure}
\includegraphics[width=14cm]{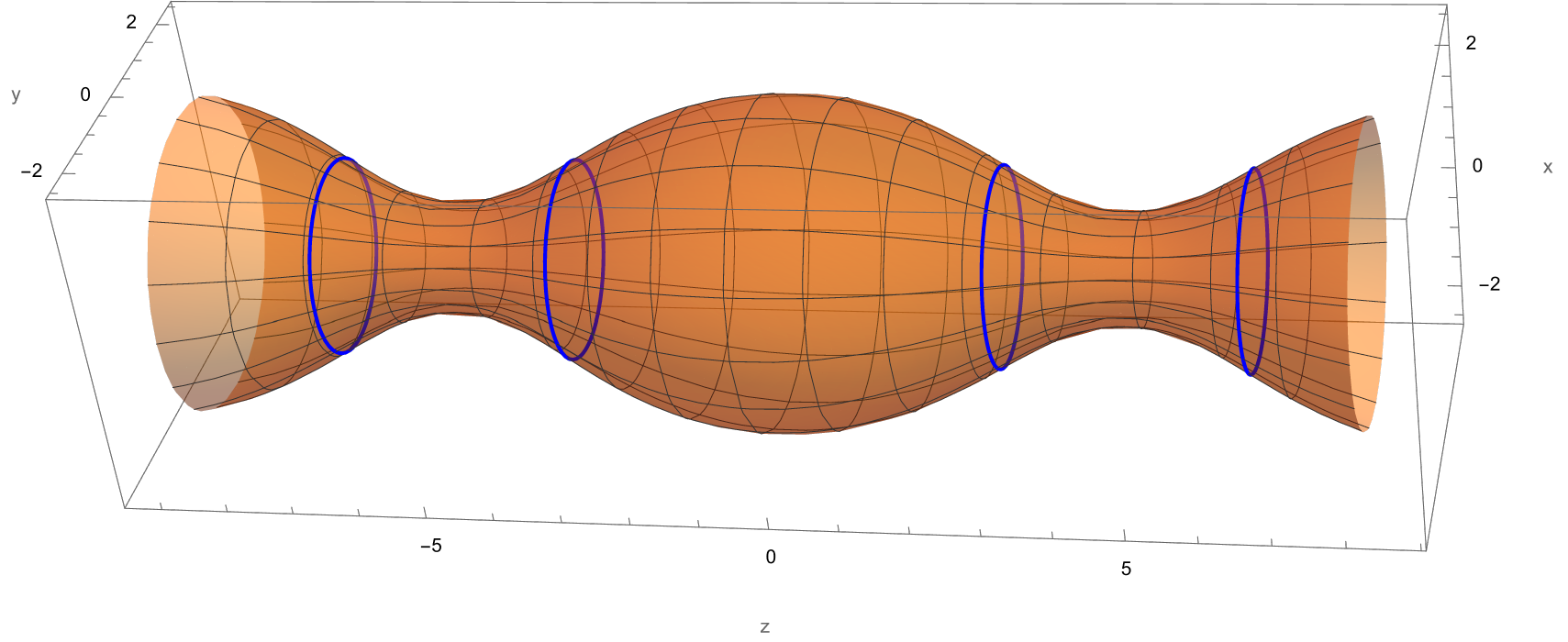}
\caption{Unduloid with parameters $\alpha = 0.8$ and $\gamma = 2.4$. The circles indicate where the curvature of the generating curve vanishes ($v \approx -4.22, -0.78, 5.81, 9.26$). Any domain bounded by two of these circles can be considered to construct a conformally extremal metric.}
\label{unduloid}
\end{figure}

\end{proof}

\end{document}